\documentclass{amsart}
\usepackage{shortsalch, xy}
\xyoption{all}
\title{The Bousfield localizations and colocalizations of the discrete model structure.}
\date{July 2016.}
\begin{document}
\begin{abstract}
We compute the Bousfield localizations and Bousfield colocalizations of discrete model categories, including the homotopy categories and the algebraic $K$-groups of these localizations and colocalizations. We prove necessary and sufficient conditions for a subcategory of a category to appear as the subcategory of fibrant objects for some such model structure. We also prove necessary and sufficient conditions for a monad to be the fibrant replacement monad of some such model structure.
\end{abstract}

\maketitle

\section{Introduction.}

This short paper answers some natural questions about model categories. 
\begin{questions}\label{motivating question}
\begin{enumerate}
\item Given a category $\mathcal{C}$ and a reflective subcategory $\mathcal{A}$ of $\mathcal{C}$, is there some model structure on $\mathcal{C}$ in which $\mathcal{A}$ is the subcategory of fibrant objects, and such that the reflector functor is fibrant replacement?
\item Given an idempotent monad on a category $\mathcal{C}$, is there some model structure on $\mathcal{C}$ such that the monad is a fibrant replacement monad? 
\item The simplest model structure on a category is the discrete model structure, that is, the model structure in which all maps are cofibrations and fibrations, and the weak equivalences are the isomorphisms.
What are all the Bousfield localizations and Bousfield colocalizations of the discrete model structure?
\item What is the homotopy category of each of those Bousfield localizations and Bousfield colocalizations? 
\item Similarly: which categories are equivalent to the homotopy category of some Bousfield localization or Bousfield colocalization of the discrete model structure on $\mathcal{C}$?
\item What are the algebraic $K$-groups (in the sense of Waldhausen, as in~\cite{MR802796}) of each of those Bousfield localizations and Bousfield colocalizations?
\end{enumerate}
\end{questions}
None of the answers to these questions are difficult to prove, but with the exception of the question about idempotent monads, the answers to these questions apparently do not already appear in the literature; and from these questions repeatedly coming up in dialogue with other homotopy theorists, the answers to these questions do not seem to be as well-known as they ought to be. So we feel that this little paper adds something of use to the literature; see Remark~\ref{remark on why bother}.

Before we give the answers to these questions, 
here is a quick account of what model categories are for, and why you might care about them: a {\em model structure on a category} (originally due to Quillen, in~\cite{MR0223432}; see Conventions~\ref{def of model cat} for a definition, and~\cite{MR1361887} for an excellent introductory account) is the minimal
structure one needs on a category in order to make the usual constructions of homotopy theory. For example, the classical homotopy theory of topological spaces is described by the Serre model structure on topological spaces. There are many other examples of model categories, including many algebraic examples: for example, categories of chain complexes over rings, whose homotopy categories recover derived categories of rings.

Model categories also have a very elegant and powerful theory of localization, {\em Bousfield localization}, for which a good reference is~\cite{MR1944041}. Given a model category $\mathcal{C}$, a {\em Bousfield localization of $\mathcal{C}$} is a model category $\mathcal{C}^{\prime}$ whose underlying category is the same as that of $\mathcal{C}$, whose cofibrations are the same as those of $\mathcal{C}$, and whose collection of weak equivalences contains the collection of weak equivalences of $\mathcal{C}^{\prime}$. Dually, a {\em Bousfield colocalization of $\mathcal{C}$} (also called {\em right Bousfield localization}) is a model category $\mathcal{C}^{\prime}$ whose underlying category is the same as that of $\mathcal{C}$, whose fibrations are the same as those of $\mathcal{C}$, and whose collection of weak equivalences contains the collection of weak equivalences of $\mathcal{C}^{\prime}$.

Now here are the answers we get to Questions~\ref{motivating question}. 
All of these answers are immediate consequences of Theorem~\ref{main thm}, with the exception of the statement on algebraic $K$-groups, which is Theorem~\ref{k-thy thm}. 
\begin{answers}\label{the answers}
Since the underlying category of any model category is finitely complete and finitely co-complete, to get a positive answer to any of the Questions~\ref{motivating question}, $\mathcal{C}$ must have all finite limits and finite colimits; so assume that $\mathcal{C}$ has all finite limits and finite colimits. Suppose furthermore that $\mathcal{C}$ is finitely well-complete, that is, $\mathcal{C}$ admits all intersections of strong monomorphisms; this condition, as well as its dual, holds in the familiar concrete categories of everyday mathematical practice (see Remarks~\ref{remark on fwc property} and~\ref{remark on co-fwc prop} for some explanation and examples of this).
\begin{enumerate}
\item 
Given a reflective subcategory $\mathcal{A}$ of $\mathcal{C}$, clearly $\mathcal{A}$ can only be the subcategory of fibrant objects of $\mathcal{C}$ if $\mathcal{A}$ is replete, that is, closed under isomorphisms. 
The converse is also true: given a replete reflective subcategory of $\mathcal{C}$, there exists a model structure---specifically, a Bousfield localization of the discrete model structure---whose subcategory of fibrant objects is $\mathcal{A}$, and for which any reflector functor $\mathcal{C}\rightarrow \mathcal{A}$ is a fibrant replacement functor. 
\item Suppose that $\mathbb{T}$ is an idempotent monad on $\mathcal{C}$. 
There does indeed exist a model structure on $\mathcal{C}$ such that $\mathbb{T}$ is a fibrant replacement monad for that model structure. (This particular fact can also be proven using the Bousfield-Friedlander construction of model structures from Quillen idempotent monads, as in Theorem~A.7 of~\cite{MR513569} and Theorem~9.3 of~\cite{MR1814075}.)
\item 
The collection of Bousfield localizations of $\mathcal{C}$
is in bijection with the collection of replete reflective subcategories of $\mathcal{C}$. This bijection reverses the ordering, in the sense that, if $\mathcal{M},\mathcal{M}^{\prime}$ are Bousfield localizations of the discrete model structure on $\mathcal{C}$ and $\mathcal{M}^{\prime}$ is a Bousfield localization of $\mathcal{M}$, then the replete reflective subcategory of $\mathcal{C}$ associated to $\mathcal{M}$ contains the replete reflective subcategory of $\mathcal{C}$ associated to $\mathcal{M}^{\prime}$.

There are similar bijections between each of the above two partially-ordered collections and the partially-ordered collection of natural equivalence classes of idempotent monads on $\mathcal{C}$; see Theorem~\ref{main thm} for the rigorous statement.
\item Given a Bousfield localization of the discrete model structure on $\mathcal{C}$, its homotopy category is equivalent to the full subcategory of $\mathcal{C}$ generated by the fibrant objects.
\item Given a category $\mathcal{X}$, we can produce a Bousfield localization of the discrete model structure on $\mathcal{C}$ whose homotopy category is equivalent to $\mathcal{X}$ if and only if $\mathcal{X}$ is equivalent to some reflective subcategory of $\mathcal{C}$.
\item If $\mathcal{C}$ is pointed, then any Bousfield localization of the discrete model structure makes $\mathcal{C}$ into a Waldhausen category, and the
Waldhausen $K$-groups $K_*(\mathcal{C}) = \pi_*\left( \Omega \left| wS_{\bullet}\mathcal{C}\right| \right)$ are all trivial.
\end{enumerate}
\end{answers}
The arguments used in this paper dualize without difficulties; in particular, we do not make any use of the small object argument or any cofibrant generation hypotheses. Hence the dual statements of Answers~\ref{the answers} also hold, if we assume that $\mathcal{C}$ is finitely co-well-complete.

As an amusing application, we have Corollary~\ref{specific application}:
let $R\rightarrow S$ be a morphism of commutative rings, and equip $\Mod(R)$ with the discrete model structure. Then there exists a Bousfield localization of 
$\Mod(R)$ for which the base change functor $-\otimes_R S$ is a fibrant replacement monad
if and only if the natural multiplication map $S\otimes_R S \rightarrow S$
is an isomorphism.

\begin{remark}\label{remark on why bother}
A lot of the work of proving the results in this paper is done by the construction of a factorization system from a reflective subcategory from the paper~\cite{MR779198} (see also the more recent paper~\cite{MR2369170} of Rosick\'{y} and Tholen, in which the ideas of \cite{MR779198} were revisited and expanded upon); we recall that result in Theorem~\ref{chk thm}. 
With the aid of that theorem, none of the statements in Answers~\ref{the answers} are difficult to prove. 
We wrote an early version of this paper in December 2011, but we were not sure if the results were already known. (The paper has been completely rewritten and simplified since then.)
After encouragement from several other homotopy theorists who work with model categories and found the results of this paper novel, we decided to submit this paper
for publication, since neither \cite{MR779198} nor \cite{MR2369170} give clean results in 
familiar homotopy-theoretic language (instead of model categories and 
Bousfield localization, 
\cite{MR779198} and \cite{MR2369170} work with torsion theories and factorization systems), so one does not find the classification of the Bousfield localizations
of the discrete model structure on a category, or the computation of 
the homotopy category
and algebraic $K$-theory of all such localizations, in the existing literature. So we feel that the present paper
adds something worthwhile to the literature on the subject.
\end{remark}

We are grateful to C. Weibel for finding errors in two lemmas in an earlier version of this paper.

\begin{convention}\label{def of model cat}
It is worth being clear on some definitions which are not totally consistent in the literature:
\begin{itemize}
\item In this paper, a ``model category'' is a 
closed model category in the sense of Quillen's monograph~\cite{MR0223432}, that is (briefly): a finitely complete, finitely co-complete 
category equipped with three closed-under-retracts classes of morphisms called ``cofibrations,'' ``weak equivalences,'' and ``fibrations,'' such that each morphism factors into an acyclic cofibration followed by a fibration, each morphism factors as a cofibration followed by an acyclic fibration, weak equivalences have the two-out-of-three property, the cofibrations are precisely the maps which have the left lifting property with respect to acyclic fibrations, and the acyclic cofibrations are precisely the maps which have the left lifting property with respect ot the fibrations. We do {\em not} assume that a model category has all limits and colimits, we do {\em not} include a specific choice of factorization systems as part of the structure of a model category, and we do {\em not} assume that the factorization systems are functorial. (We will show, in fact, that in the model categories considered in this paper, factorization systems can be found which {\em are} functorial.)
\item In this paper, a subcategory $\mathcal{A}$ of a category $\mathcal{C}$ is said to be {\em reflective} if the inclusion functor $\mathcal{A}\hookrightarrow \mathcal{C}$ {\em is full} and admits a left adjoint. (Not all sources include fullness in the definition of a reflective subcategory; see for example~\cite{MR1712872}.)
\end{itemize}
\end{convention}

\section{From Bousfield localizations of discrete model structures to replete reflective subcategories.}

\begin{lemma}\label{when right-homotopic maps are left-homotopic}
Let $\mathcal{C}$ be a model category which is a Bousfield localization of a discrete model category. 
\begin{enumerate}
\item A map in $\mathcal{C}$ is an acyclic fibration if and only if it is an isomorphism.
\item Two maps in $\mathcal{C}$ with fibrant codomain are left-homotopic if and only if they are
equal.
\item Two maps in $\mathcal{C}$ with fibrant codomain are right-homotopic if and only if they are equal.
\end{enumerate}
\end{lemma}
\begin{proof}
\begin{enumerate}
\item 
Bousfield localization does not change the cofibrations of the underlying model category, hence it does not change the acyclic fibrations (since they
are determined by the cofibrations). In a discrete model category, the acyclic fibrations are the isomorphisms; hence in any localization of a
discrete model category, the acyclic fibrations are still the isomorphisms.
\item
Let $f,g:A\stackrel{}{\longrightarrow} B$ be two maps in $\mathcal{C}$,
and suppose that $B$ is fibrant. Then any two maps that are left-homotopic are
left-homotopic by a ``very good'' cylinder object (see~\cite{MR1361887}, particularly
Lemma~4.6), so if we assume that $f,g$ are left-homotopic, then
we have a map
$\Cyl(A)\stackrel{h}{\longrightarrow} B$ 
and a factorization
\[ A\coprod A\stackrel{i}{\longrightarrow} \Cyl(A)\stackrel{j}{\longrightarrow} A\]
of the codiagonal map $A\coprod A\rightarrow A$, with $i$ a cofibration and $j$ an acyclic fibration,
and $h\circ i = f\coprod g: A\coprod A\rightarrow B$. But every acyclic fibration in $\mathcal{C}$ is an isomorphism;
so the map $j$ is an isomorphism, so we can factor $f\coprod g$ all the way through the codiagonal map $A\coprod A\rightarrow A$. So $f=g$.
\item Maps with fibrant codomain and cofibrant domain are left-homotopic if and only if they are right-homotopic; see Lemma 4.21 in \cite{MR1361887} a proof of this.
The statement then follows immediately from the previous parts of this lemma.
\end{enumerate}
\end{proof}

\begin{lemma} \label{two maps between fibrant replacements are right-homotopic}
Let $\mathcal{C}$ a model category and let $A\stackrel{f}{\longrightarrow} B$ be a map in $\mathcal{C}$.
Let $g_1,g_2$ be two maps making the diagram
\[\xymatrix{ A\ar[d]^f \ar[r]^{cof, we} & PA 
 \ar[d]^{g_i}
 \ar[r]^{fib} & 1\ar[d] \\
 B\ar[r]^{cof, we} & PB\ar[r]^{fib} & 1}\]
commute for $i=1,2$, 
where maps labelled $cof, we, fib$ are cofibrations, weak equivalences, and fibrations in $\mathcal{C}$, respectively.
Then $g_1$ and $g_2$ are right-homotopic.
\end{lemma}
\begin{proof}
We have the factorization of the diagonal map
\[ PB\stackrel{cof, we}{\longrightarrow} \Path(PB) \stackrel{fib}{\longrightarrow} PB\times PB\]
and we arrange this composable pair of maps as well as the given maps into 
the commutative diagram:
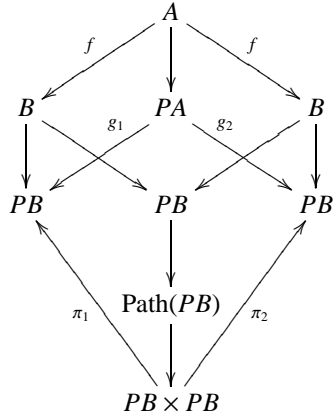
\[\xymatrix{
 & A \ar[d]\ar[rd]^f \ar[ld]_f & \\
B \ar[rd] \ar[d] & PA\ar[ld]_(0.3){g_1}\ar[rd]^(0.3){g_2} & B\ar[ld] \ar[d] \\
PB & PB\ar[d] & PB \\
 & \Path(PB)\ar[d] & \\
 & PB\times PB\ar[uul]^{\pi_1} \ar[uur]_{\pi_2} 
}\]
and, redrawing a portion of the central axis of this commutative diagram, we have
\[\xymatrix{ A \ar[r]^{cof, we} \ar[d] & PA\ar[d] \\
\Path(PB) \ar[r]^{fib} & PB\times PB.}\]
So there exists a lift $PA\rightarrow \Path(PB)$ of this acyclic cofibration over this fibration. This lift is a right homotopy of $g_1$ and $g_2$.
\end{proof}

\begin{prop}\label{fibrant replacement functor exists}
Let $\mathcal{C}$ be a model category which is a Bousfield localization of a discrete model category.
Then, given a map $A\stackrel{f}{\longrightarrow} B$ in $\mathcal{C}$, there exists a unique map $g$ making the following diagram commute:
\begin{equation}\label{fibrant replacement mapping diagram} \xymatrix{ A\ar[d]^f \ar[r]^{cof, we} & PA\ar[d]^g\ar[r]^{fib} & 1 \ar[d]\\
 B\ar[r]^{cof, we} & PB\ar[r]^{fib} & 1,}\end{equation}
where maps labelled $cof, we, fib$ are cofibrations, weak equivalences, and fibrations in $\mathcal{C}$, respectively. 
As a consequence, $\mathcal{C}$ admits functorial fibrant replacement. More specifically, every choice of fibrant replacements in $\mathcal{C}$ extends
to a fibrant replacement functor.
\end{prop}
\begin{proof}
Existence of the desired map is given trivially by the lifting property of an acyclic cofibration over a fibration, and the commutative square
\[ \xymatrix{ 
 A \ar[r]^{cof, we} \ar[d] & PA\ar[d] \\
 PB \ar[r]^{fib} & 1. }\]

The real question is why such a map $PA\rightarrow PB$ is unique.
Let $g,g^{\prime}: PA\rightarrow PB$ be two maps making the above diagram commute. 
By Lemma~\ref{two maps between fibrant replacements are right-homotopic}, $g$ and $g^{\prime}$ are right-homotopic. 
By Lemma~\ref{when right-homotopic maps are left-homotopic}, $g$ and $g^{\prime}$ are hence left-homotopic (since they have
fibrant codomain), hence equal. This implies that any two fibrant replacements for an object in $\mathcal{C}$ are {\em uniquely} isomorphic.

Since the map $g$ filling in diagram~\ref{fibrant replacement mapping diagram} is uniquely determined by $f$, we will write $Pf$ instead of $g$ for this map.
It is easy to check that the assignment $f\mapsto Pf$ preserves composition and identity maps; hence $P$ is a functor $\mathcal{C}\rightarrow \mathcal{C}$ landing in 
the full subcategory generated by the fibrant objects. So $P$ is a fibrant replacement functor.
\end{proof}

\begin{prop}\label{bousfield loc model structures have refl fibrant replacement} Let $\mathcal{C}$ be a model category which is a Bousfield localization of a discrete model category.
Let $\mathcal{A}$ be the full subcategory of $\mathcal{C}$ generated by the fibrant objects. 
Let $G$ be the inclusion functor
$\mathcal{A}\stackrel{G}{\longrightarrow}\mathcal{C}$ and let $P$ be the fibrant replacement functor 
$\mathcal{C}\stackrel{P}{\longrightarrow} \mathcal{A}$ constructed in Proposition~\ref{fibrant replacement functor exists}. 
Then $P$ is left adjoint to $G$. Hence $\mathcal{A}$ is replete and reflective.
\end{prop}
\begin{proof}
Since $G$ is a full faithful functor, $P$ is left adjoint to $G$ if and only if,
for every pair of objects $A,B$ of $\mathcal{C}$ with $B$ fibrant, every map $A\rightarrow B$ factors uniquely through $A\rightarrow PA$.
But by Proposition~\ref{fibrant replacement functor exists}, there exists a unique map $g$ filling in the diagram
\begin{equation} \xymatrix{ A\ar[d]^f \ar[r]^{cof, we} & PA\ar[d]^g\ar[r]^{fib} & 1\ar[d] \\
 B\ar[r]^{cof, we}_{\cong} & PB\ar[r]^{fib} & 1.}\end{equation}
This is the desired unique factor map. Hence $P$ is left adjoint to $G$. It is immediate from its definition that $\mathcal{A}$ is also replete.
\end{proof}

\section{From replete reflective subcategories to Bousfield localizations of discrete model structures.}

To produce a model structure from a reflective subcategory, we will use a theorem of Cassidy, H\'{e}bert, and Kelly from~\cite{MR779198}.
First we give the necessary definitions, Definitions~\ref{def of fwc} and~\ref{def of fs}, which are fairly well-known:
\begin{definition}\label{def of fwc}
Let $\mathcal{C}$ be a category.
\begin{itemize}
\item Given a collection $\mathcal{E}$ of morphisms in $\mathcal{C}$, we write $\mathcal{E}^{\downarrow}$ for the collection of morphisms in $\mathcal{C}$ with the right lifting property with respect to all maps in $\mathcal{E}$, i.e., $\mathcal{E}^{\downarrow}$ is the collection of maps $f: X \rightarrow Y$ in $\mathcal{C}$ with the property that,
for every commutative square
\begin{equation}\label{comm sq 4308}\xymatrix{
 V \ar[d]_g \ar[r] & X \ar[d]^f \\ W \ar[r] & Y 
}\end{equation}
in $\mathcal{C}$ with $g$ in $\mathcal{E}$, there exists a map $W \rightarrow X$ filling in diagram~\ref{comm sq 4308} and making it commute.
\item Dually, $\mathcal{E}^{\uparrow}$ is the collection of maps with the left lifting property with respect to all maps in $\mathcal{E}$.
\item The {\em strong monomorphisms in $\mathcal{C}$} is the collection of maps $\mathcal{E}^{\downarrow}$, where $\mathcal{E}$ is the collection of all epimorphisms in $\mathcal{C}$.
\item We say that $\mathcal{C}$ is {\em finitely well-complete} if $\mathcal{C}$ has finite limits and intersections of all families of strong monomorphisms, i.e., $\mathcal{C}$ has finite limits and, for every collection (not necessarily a set!) of strong monomorphisms $\{ X_i \stackrel{f_i}{\longrightarrow} Y\}_{i\in I}$ in $\mathcal{C}$, the pullback of the family of maps $f_i$ exists.
\end{itemize}
\end{definition}

\begin{remark}\label{remark on fwc property}
Remember that a category $\mathcal{C}$ is said to be {\em well-powered} if the collection of subobjects of each object forms a set; this condition is satisfied by the usual familiar concrete categories (sets, groups, topological spaces, \dots ). If $\mathcal{C}$ is complete and well-powered, then it is easy to see that $\mathcal{C}$ is finitely well-complete. Consequently most familiar categories used in day-to-day mathematical practice are finitely well-complete. There are also many finitely well-complete categories even aside from those that are complete and well-powered; for example, the category of finite sets (this example, and useful discussion, is given in~\cite{MR779198}).
\end{remark}

\begin{definition}\label{def of fs}
Given a category $\mathcal{C}$, a {\em factorization system on $\mathcal{C}$} is a pair $(\mathcal{E},\mathcal{M})$ of collections of morphisms in $\mathcal{C}$ satisfying the axioms:
\begin{itemize}
\item Given a morphism $f: X \rightarrow Y$ in $\mathcal{C}$, there exists some object $Z$, a map $g: X \rightarrow Z$ in $\mathcal{E}$, and a map $h: Z \rightarrow Y$ in $\mathcal{M}$, such that $h\circ g = f$.
\item We have equalities $\mathcal{E}^{\downarrow} = \mathcal{M}$ and $\mathcal{E} = \mathcal{M}^{\uparrow}$. 
\end{itemize}
\end{definition}
Definition~\ref{def of fs} is sometimes phrased slightly differently: for example, in~\cite{MR779198}, the containment $\mathcal{M}^{\uparrow} \subseteq \mathcal{E}$ is given instead of the equalities $\mathcal{E}^{\downarrow} = \mathcal{M}$ and $\mathcal{E} = \mathcal{M}^{\uparrow}$, but it is promptly (see the discussion at the beginning of section~2 in~\cite{MR779198}) explained that $\mathcal{E}^{\downarrow} = \mathcal{M}$ and $\mathcal{E} = \mathcal{M}^{\uparrow}$ are implied by such a definition.

\begin{theorem} {\bf (Cassidy-H\'{e}bert-Kelly.)}\label{chk thm}
Suppose $\mathcal{C}$ is a finitely well-complete category, and let $\mathcal{A}$ be a reflective subcategory of $\mathcal{C}$. Let $F: \mathcal{C}\rightarrow \mathcal{A}$ denote a reflector functor.  Then there exists a (necessarily unique) factorization system $(\mathcal{E},\mathcal{M})$ on $\mathcal{C}$ with $\mathcal{E}$ the collection of maps $f$ in $\mathcal{C}$ such that $Ff$ is an isomorphism.
\end{theorem}

Now we can prove the main theorem of this subsection:
\begin{prop}\label{construction of model structure}
Let $\mathcal{C}$ be a finitely co-complete, finitely well-complete category, and let $\mathcal{A}$ be a replete reflective subcategory of $\mathcal{C}$ with reflector functor $F: \mathcal{C}\rightarrow\mathcal{A}$.
Then there exists a model structure on $\mathcal{C}$ in which:
\begin{itemize}
\item all maps are cofibrations (i.e., this model structure is a Bousfield localization of the discrete model structure),
\item the weak equivalences are the maps $f$ such that $Ff$ is an isomorphism, and
\item the fibrations are the maps with the right lifting property with respect to the weak equivalences.
\end{itemize}
\end{prop}
\begin{proof}
We check the axioms for being a model category (see Convention~\ref{def of model cat}):
\begin{itemize}
\item By assumption, $\mathcal{C}$ is finitely complete and finitely co-complete. 
\item It is a very easy exercise to show that the weak equivalences satisfy the two-out-of-three property.
\item It is trivial that cofibrations are closed under retracts. It is a very easy exercise to show that, in any category:
\begin{itemize}\item
a retract of an isomorphism is an isomorphism, and
\item a retract of a map with the right lifting property with respect to a collection of morphisms also has the right lifting property with respect to that collection of morphisms.\end{itemize}
Consequently weak equivalences and fibrations are each closed under retracts. 
\item Suppose $f: X\rightarrow Y$ is an acyclic fibration. Then $f$ belongs to both the collection of weak equivalences and the collection of maps which have the right lifting property with respect to the weak equivalences. It is a standard (and easy) exercise to show that the lifting property produces an inverse for $f$, so that $f$ is an isomorphism. Consequently the cofibrations are exactly the maps with the left lifting property with respect to the acyclic fibrations.
\item From our definitions, it is immediate that the acyclic cofibrations are exactly the maps with the left lifting property with respect to the fibrations.
\item Every map $f: X \rightarrow Y$ in $\mathcal{C}$ trivially factors as a cofibration (namely, $f$) followed by an acyclic fibration (namely, $\id_Y$).
\item Theorem~\ref{chk thm} gives us that every map in $\mathcal{C}$ factors as an acyclic cofibration followed by a fibration.
\end{itemize}
\end{proof}

\section{Conclusions.}

Now we give a definition of idempotent monads. This definition is equivalent to the many other well-known definitions of an idempotent monad; see for example Proposition~4.2.3 of~\cite{MR1313497}.
\begin{definition}
Let $\mathcal{C}$ be a category. A monad $T: \mathcal{C} \rightarrow \mathcal{C}, \eta: \id \rightarrow T, \mu: T\circ T \rightarrow T$ on $\mathcal{C}$ is said to be {\em idempotent} if, for all objects $X$ of $\mathcal{C}$, both maps $T\eta X,\eta TX: TX \rightarrow TTX$ are isomorphisms.
Two idempotent monads $(T,\eta,\mu),(T^{\prime},\eta^{\prime},\mu^{\prime})$ on $\mathcal{C}$ are {\em naturally equivalent} if there is a morphism of monads $(T,\eta,\mu)\rightarrow (T^{\prime},\eta^{\prime},\mu^{\prime})$ between them which is a natural equivalence of the underlying functors $T,T^{\prime}$.
\end{definition}

\begin{definition}
Suppose $\mathcal{C}$ is a finitely complete, finitely co-complete category.
We define three partially-ordered collections (not necessarily sets!) as follows:
\begin{itemize}
\item let $\Refl(\mathcal{C})$ be the collection of replete reflective subcategories of $\mathcal{C}$, ordered by inclusion,
\item let $\Loc(\mathcal{C})$ denote the collection of Bousfield localization model structures of the discrete model structure on $\mathcal{C}$, ordered by letting $X\leq Y$ if and only if the model structure $Y$ is a Bousfield localization of the model structure $X$,
\item and let $\Idem\Monads(\mathcal{C})$ denote the collection of natural equivalence classes of idempotent monads on $\mathcal{C}$, ordered by letting $[S]\leq [T]$ if and only if there exists a morphism of monads $S \rightarrow T$. (It is very easy to check that this ordering is well-defined.) 
\end{itemize}
\end{definition}
It follows from Proposition~\ref{bousfield loc model structures have refl fibrant replacement} that every model structure in $\Loc(\mathcal{C})$ has {\em functorial} fibrant replacement, even though we do not assume functoriality of factorization systems in our definition of a model category (see Convention~\ref{def of model cat}).

\begin{theorem}\label{main thm}
Let $\mathcal{C}$ be a finitely co-complete, finitely well-complete category. 
Then there exist order-preserving bijections
\[ \Refl(\mathcal{C})^{\op} \cong \Loc(\mathcal{C}) \cong \Idem\Monads(\mathcal{C}).\]
These bijections have the following properties:
\begin{itemize}
\item Given a reflective replete subcategory $\mathcal{A}\in \Refl(\mathcal{C})$ with reflector functor $F:\mathcal{C}\rightarrow\mathcal{A}$, its associated model structure in $\Loc(\mathcal{C})$ has the following properties:
 \begin{itemize}
 \item all maps are cofibrations,
 \item the weak equivalences are the maps $f$ such that $Ff$ is an isomorphism,
 \item the fibrations are the maps with the right lifting property with respect to the weak equivalences,
 \item the fibrant objects are exactly the objects of $\mathcal{A}$,
 \item $F$ is a fibrant replacement functor for this model structure (and in fact all fibrant replacement functors for this model structure are naturally isomorphic to $F$), and 
 \item the homotopy category of this model structure is equivalent to $\mathcal{A}$.
 \end{itemize}
 The idempotent monad associated to $\mathcal{A}$ is simply the monad associated to the adjunction $F\dashv G$, where $G: \mathcal{A}\hookrightarrow\mathcal{C}$ is the inclusion functor.
\item Given a model structure $X$ in $\Loc(\mathcal{C})$, its associated reflective replete subcategory $\mathcal{A}$ of $\mathcal{C}$ is simply the full subcategory generated by the fibrant objects of $\mathcal{C}$, with reflector functor given by fibrant replacement. 
\item Given a natural equivalence class of idempotent monads on $\mathcal{C}$, choose an idempotent monad  $(T,\eta,\mu)$ in the class. The reflective replete subcategory of $\mathcal{C}$ associated to that natural equivalence class is the essential image of $T$, with reflector functor given by $T$. Its associated model structure in $\Loc(\mathcal{C})$ has the following properties:
 \begin{itemize}
 \item all maps are cofibrations,
 \item the weak equivalences are the maps $f$ such that $Tf$ is an isomorphism,
 \item the fibrations are the maps with the right lifting property with respect to the weak equivalences,
 \item the fibrant objects are those in the essential image of $T$,
 \item $T$ is a fibrant replacement functor for this model structure (and in fact all fibrant replacement functors for this model structure are naturally isomorphic to $T$), and 
 \item the homotopy category of this model structure is equivalent to the essential image of $T$.
 \end{itemize}
\end{itemize}
\end{theorem}
\begin{proof}
The equivalence between idempotent monads and reflective subcategories is classical; see e.g. Proposition~4.3.2 of~\cite{MR1313497}, or~\cite{MR0210125}. Briefly, to each idempotent monad $(T,\eta,\mu)$ one associates the reflective category consisting of the image of $T$, and to each reflective category one associates the idempotent monad given by the adjunction of the reflector and the inclusion of the subcategory. Each natural equivalence class of idempotent monads contains an idempotent monad whose image is replete; consequently we get an order-preserving bijection between the replete reflective subcategories and the natural equivalence classes of idempotent monads.

The equivalence between $\Refl(\mathcal{C})^{\op}$ and $\Loc(\mathcal{C})$
is given by Proposition~\ref{construction of model structure} and Proposition~\ref{bousfield loc model structures have refl fibrant replacement}.  Beginning with a reflective replete subcategory $\mathcal{A}$ of $\mathcal{C}$, the associated model structure (constructed in Proposition~\ref{construction of model structure}) has fibrant objects the objects of $\mathcal{A}$; the reflective subcategory of $\mathcal{C}$ (constructed in Proposition~\ref{bousfield loc model structures have refl fibrant replacement}) associated to this model structure is $\mathcal{A}$ again. 

Similarly, if we begin with a model structure $\mathcal{M}$ in $\Loc(\mathcal{C})$, let $\mathcal{M}^{\prime}$ denote the model structure on $\mathcal{C}$ associated (by Proposition~\ref{construction of model structure}) to the reflective subcategory of $\mathcal{C}$ associated (by Proposition~\ref{bousfield loc model structures have refl fibrant replacement}) to $\mathcal{M}$. Clearly all maps are cofibrations in both model structures.
Let $F$ denote a reflector functor from $\mathcal{C}$ to the full subcategory generated by the fibrant objects, and let $\eta: \id \rightarrow F$ denote the natural transformation of the reflector. The weak equivalences in $\mathcal{M}^{\prime}$ are the maps inverted by $F$. If $f: X \rightarrow Y$ is a weak equivalence in $\mathcal{M}$, then we have the commutative square
\begin{equation}\label{comm sq 43438}\xymatrix{
 X \ar[r]^f \ar[d]_{\eta X} & Y \ar[d]^{\eta Y} \\
 FX \ar[r]_{Ff} & FY,
}\end{equation}
and since $f$ and $\eta X$ and $\eta Y$ are all weak equivalences, the two-out-of-three property implies that $Ff$ is also a weak equivalence. Hence $Ff$ is a weak equivalence between fibrant-cofibrant objects, hence (see e.g. Lemma~4.24 of~\cite{MR1361887}) there exists an inverse for $Ff$ up to left (equivalently, right) homotopy in the model structure $\mathcal{M}$. Now Lemma~\ref{when right-homotopic maps are left-homotopic} implies that there exists an inverse for $Ff$ ``on the nose,'' not just up to homotopy. Hence $f$ is inverted by $F$, hence the weak equivalences in $\mathcal{M}$ are contained in the weak equivalences in $\mathcal{M}^{\prime}$. 

Conversely, if $f: X \rightarrow Y$ is a weak equivalence in $\mathcal{M}^{\prime}$,
then the commutativity of the square~\ref{comm sq 43438}, together with $Ff$ being an isomorphism, $\eta X$ and $\eta Y$ being weak equivalences in $\mathcal{M}$, and the two-out-of-three property of weak equivalences, implies that $f$ is a weak equivalence in $\mathcal{M}$. Hence the weak equivalences in $\mathcal{M}^{\prime}$ are contained in the weak equivalences in $\mathcal{M}$.

So $\mathcal{M}$ and $\mathcal{M}^{\prime}$ have the same cofibrations and the same weak equivalences; since the fibrations are determined by the acyclic cofibrations (see Convention~\ref{def of model cat}), we now know that $\mathcal{M} = \mathcal{M}^{\prime}$, and so the function
$\Refl(\mathcal{C})^{\op} \rightarrow \Loc(\mathcal{C})$ constructed in Proposition~\ref{construction of model structure} and the function $\Loc(\mathcal{C}) \rightarrow \Refl(\mathcal{C})^{\op}$ constructed in Proposition~\ref{bousfield loc model structures have refl fibrant replacement} are mutually inverse.

That the homotopy category of any model structure in $\Loc(\mathcal{C})$ is equivalent to its associated reflective replete subcategory is a consequence of Lemma~\ref{two maps between fibrant replacements are right-homotopic}: the homotopy category of the model structure is equivalent to the category obtained by taking the full subcategory of $\mathcal{C}$ generated by the fibrant-cofibrant objects, and imposing the equivalence relation on maps given by left, equivalently right, homotopy. But by Lemma~\ref{two maps between fibrant replacements are right-homotopic}, for model structures in $\Loc(\mathcal{C})$, two maps with fibrant codomain are right-homotopic if and only if they are equal. Hence $\Ho(\mathcal{C})$ is equivalent to the full subcategory of $\mathcal{C}$ generated by the fibrant objects.
\end{proof}

Here is a specific application to the category of modules over a commutative ring $R$.
\begin{corollary}\label{specific application}
Let $R$ be a commutative ring and let $\Mod(R)$ be the category of $R$-modules, equipped
with the discrete model structure.
Let $S$ be a commutative $R$-algebra. Then there exists a Bousfield localization of 
$\Mod(R)$ in which the base change functor $-\otimes_R S$ is a fibrant replacement monad
if and only if the natural multiplication map $S\otimes_R S \rightarrow S$
is an isomorphism.
\end{corollary}
\begin{proof}
By Theorem~\ref{main thm}, in order for the base change functor to be a fibrant replacement,
the forgetful functor from $\Mod(S)$ to $\Mod(R)$ must be a full functor.
Since it is clearly faithful, its fullness is equivalent to the counit
$M\otimes_R S\rightarrow M$ of the adjunction being an isomorphism for every $S$-module $M$ (see for example Theorem~IV.3.1 of~\cite{MR1712872}). 
This counit is simply the tensor product of $M$ with the multiplication map 
$S\otimes_R S\stackrel{\nabla}{\longrightarrow} S$. 
So, for the counit to be an isomorphism, it is necessary and sufficient that
$\nabla$ be an isomorphism.
\end{proof}

\begin{remark}\label{remark on co-fwc prop}
Every argument used in the proof of Theorem~\ref{main thm} dualizes; in particular, we did not use a cofibrant generation assumption or small object argument anywhere. Consequently, if $\mathcal{C}$ is finitely complete and finitely co-well-complete, then Theorem~\ref{main thm} also gives us order-preserving bijections
\begin{equation}\label{dual bijections} \left(\co\Refl(\mathcal{C})\right)^{\op} \cong \co\Loc(\mathcal{C}) \cong \Idem\co\Monads(\mathcal{C})\end{equation}
between the partially-ordered collection of coreflective replete subcategories of $\mathcal{C}$, the partially-ordered collection of Bousfield colocalizations of the discrete model structure on $\mathcal{C}$, and the partially-ordered collection of natural equivalence classes of idempotent comonads on $\mathcal{C}$.

A dual version of Remark~\ref{remark on fwc property} shows easily that a co-complete, co-well-powered category is finitely co-well-complete, so again, the bijections~\ref{dual bijections} hold in most categories of everyday practical mathematical experience (sets, groups, rings, topological spaces, etc.), and again, there are also plenty of finitely co-well-complete categores even aside from those that are co-complete and co-well-powered, for example, the category of finite sets.
\end{remark}

\begin{lemma}\label{maps between fibrant objs are fibs}
Let $\mathcal{C}$ be a model category whose model structure is a Bousfield localization of the discrete model structure. Then every map between fibrant objects in $\mathcal{C}$ is a fibration.
\end{lemma}
\begin{proof}
Write $\mathcal{A}$ for the full subcategory of $\mathcal{C}$ generated by the fibrant objects. By Proposition~\ref{bousfield loc model structures have refl fibrant replacement}, $\mathcal{A}$ is a reflective replete subcategory of $\mathcal{C}$. Write $F: \mathcal{C} \rightarrow\mathcal{A}$ for a reflector functor, $\eta: \id \rightarrow F$ for the natural transformation of the reflector, and suppose that $f: V \rightarrow W$ is a weak equivalence in $\mathcal{C}$ and $g: X \rightarrow Y$ is a map between fibrant objects in $\mathcal{C}$ fitting into a commutative square
\begin{equation}\label{comm sq 43439} \xymatrix{
 V \ar[d]_f \ar[r]^a & X \ar[d]_g \\ W \ar[r]_b & Y .
}\end{equation}
 Then we have a commutative diagram
\[\xymatrix{
 V\ar@/^/[rr]^a \ar[d]_f \ar[r]_{\eta V} & FV \ar[d]_{Ff} \ar[r]_{a^{\prime}} & X \ar[d]^g \\
 W\ar@/_/[rr]_b          \ar[r]^{\eta W} & FW            \ar[r]^{b^{\prime}} & Y 
}\]
and $Ff$ is an isomorphism by Theorem~\ref{main thm}.
It is easy to check that $a^{\prime}\circ (Ff)^{-1} \circ \eta W: W \rightarrow X$ fills in the square~\ref{comm sq 43439} and makes it commute. 

Hence maps between fibrant objects have the right lifting property with respect to all acyclic cofibrations. Hence maps between fibrant objects are fibrations.
\end{proof}

\begin{theorem}\label{k-thy thm}
Let $\mathcal{C}$ be a finitely complete, finitely co-complete category, and suppose $\mathcal{C}$ is equipped with a model structure which is a Bousfield localization {\em or} a Bousfield colocalization of the discrete model structure. Suppose furthermore that $\mathcal{C}$ is pointed, so that the model structure defines the structure of a Waldhausen category on the full subcategory of $\mathcal{C}$ generated by the cofibrant objects. Then the Waldhausen $K$-theory of $\mathcal{C}$ is trivial, that is, $K_*(\mathcal{C}) = \pi_*\left( \Omega \left| wS_{\bullet}\mathcal{C}\right|\right) \cong 0$.
\end{theorem}
\begin{proof}
First suppose that $\mathcal{C}$ is a Bousfield localization of the discrete model structure. Then all objects are cofibrant, and we have the exact (in the sense of Waldhausen) functor $1: \mathcal{C} \rightarrow \mathcal{C}$ sending every object of $\mathcal{C}$ to the terminal object, and
sending every morphism to the identity morphism on the terminal object. Furthermore, since all maps in $\mathcal{C}$ are cofibrations, we have a
cofiber sequence of exact functors $\mathcal{C}\rightarrow \mathcal{C}$:
\[ \id_\mathcal{C} \rightarrow 1 \rightarrow 1.\]
Now by Waldhausen's additivity theorem (Proposition~1.3.2 and Theorem~1.4.2 of \cite{MR802796}), the functors
\[ (wS_{\bullet}\id_{\mathcal{C}}) \vee (wS_{\bullet}1), wS_{\bullet}1 : wS_{\bullet} \mathcal{C}\rightarrow wS_{\bullet} \mathcal{C} \]
are homotopic. Hence the identity map on $wS_{\bullet}\mathcal{C}$ is nulhomotopic, hence $wS_{\bullet}\mathcal{C}$ is contractible, hence
$K_n(\mathcal{C}) = \pi_n\Omega\left| wS_{\bullet} \mathcal{C}\right| \cong 0$ for all $n$.

Now suppose that $\mathcal{C}$ is a Bousfield colocalization of the discrete model structure. By the dual of Lemma~\ref{maps between fibrant objs are fibs}, all maps between cofibrant objects in $\mathcal{C}$ are cofibrations.
So the same argument as above, using Waldhausen's additivity theorem, implies contractibility of
$wS_{\bullet} W$, hence triviality of all $K$-groups.
\end{proof}

\bibliography{/home/asalch/texmf/tex/salch}{}

\def\cprime{$'$} \def\cprime{$'$} \def\cprime{$'$} \def\cprime{$'$}
\begin{thebibliography}{10}

\bibitem{MR1313497}
Francis Borceux.
\newblock {\em Handbook of categorical algebra. 2}, volume~51 of {\em
  Encyclopedia of Mathematics and its Applications}.
\newblock Cambridge University Press, Cambridge, 1994.
\newblock Categories and structures.

\bibitem{MR1814075}
A.~K. Bousfield.
\newblock On the telescopic homotopy theory of spaces.
\newblock {\em Trans. Amer. Math. Soc.}, 353(6):2391--2426 (electronic), 2001.

\bibitem{MR513569}
A.~K. Bousfield and E.~M. Friedlander.
\newblock Homotopy theory of {$\Gamma $}-spaces, spectra, and bisimplicial
  sets.
\newblock In {\em Geometric applications of homotopy theory ({P}roc. {C}onf.,
  {E}vanston, {I}ll., 1977), {II}}, volume 658 of {\em Lecture Notes in Math.},
  pages 80--130. Springer, Berlin, 1978.

\bibitem{MR779198}
C.~Cassidy, M.~H{\'e}bert, and G.~M. Kelly.
\newblock Reflective subcategories, localizations and factorization systems.
\newblock {\em J. Austral. Math. Soc. Ser. A}, 38(3):287--329, 1985.

\bibitem{MR1361887}
W.~G. Dwyer and J.~Spali{\'n}ski.
\newblock Homotopy theories and model categories.
\newblock In {\em Handbook of algebraic topology}, pages 73--126.
  North-Holland, Amsterdam, 1995.

\bibitem{MR0210125}
P.~Gabriel and M.~Zisman.
\newblock {\em Calculus of fractions and homotopy theory}.
\newblock Ergebnisse der Mathematik und ihrer Grenzgebiete, Band 35.
  Springer-Verlag New York, Inc., New York, 1967.

\bibitem{MR1944041}
Philip~S. Hirschhorn.
\newblock {\em Model categories and their localizations}, volume~99 of {\em
  Mathematical Surveys and Monographs}.
\newblock American Mathematical Society, Providence, RI, 2003.

\bibitem{MR1712872}
Saunders Mac~Lane.
\newblock {\em Categories for the working mathematician}, volume~5 of {\em
  Graduate Texts in Mathematics}.
\newblock Springer-Verlag, New York, second edition, 1998.

\bibitem{MR0223432}
Daniel~G. Quillen.
\newblock {\em Homotopical algebra}.
\newblock Lecture Notes in Mathematics, No. 43. Springer-Verlag, Berlin-New
  York, 1967.

\bibitem{MR2369170}
Ji{\v{r}}{\'{\i}} Rosick{\'y} and Walter Tholen.
\newblock Factorization, fibration and torsion.
\newblock {\em J. Homotopy Relat. Struct.}, 2(2):295--314, 2007.

\bibitem{MR802796}
Friedhelm Waldhausen.
\newblock Algebraic {$K$}-theory of spaces.
\newblock In {\em Algebraic and geometric topology ({N}ew {B}runswick,
  {N}.{J}., 1983)}, volume 1126 of {\em Lecture Notes in Math.}, pages
  318--419. Springer, Berlin, 1985.

\end{thebibliography}
\bibliographystyle{plain}
\end{document}